\documentclass[reqno]{amsart}
 \usepackage{amssymb,amsmath,amscd,graphicx,color,epstopdf,mathtools}
\usepackage[all]{xy}
\usepackage[all]{xy}
\usepackage{color}
\usepackage{hyperref}
\usepackage{enumerate}
\usepackage{amsthm}
\usepackage{epsfig}
\usepackage[english]{babel}

\let\oldmarginpar\marginpar
\renewcommand\marginpar[1]
{\oldmarginpar{\tiny\bf \begin{flushleft} #1 \end{flushleft}}}

   \newcommand{\R}{\mathbb{R}}
   \newcommand{\C}{\mathbb{C}}
      \renewcommand{\S}{\mathbb{S}}

    \newcommand{\cL}{\mathcal{L}}

\newtheorem{proposition}{{\bf Proposition}}
\newtheorem{lemma}{{\bf Lemma}}
\newtheorem{theorem}{{\bf Theorem}}

\newtheorem{remark}{Remark}

\begin{document}

\author{David Mart\'inez Torres}
\address{Departamento de Matem\'atica Aplicada, UPM, Av. Juan de Herrera, 4, 28040 Madrid, Spain}
\email{dfmtorres@gmail.com}

\author{Marcelo Silva}
\email{marcelo.santos.math@gmail.com}

\title{Non-exactness of toric Poisson structures}

\maketitle

\begin{abstract}
 We prove that a Poisson structure on a projective toric variety which is invariant by the torus action and whose 
 symplectic leaves are the torus orbits is not exact. This is deduced from a geometric criterion for non-exactness of 
 Poisson structures with a finite number of symplectic leaves.
\end{abstract}

\section{Introduction}

The Poisson cohomology of a Poisson manifold generalizes the de Rham cohomology of a symplectic manifold. 
In general Poisson cohomology is very difficult to compute and it barely extends well-known features of the de Rham cohomology. For instance, whereas 
it is true that the Poisson 
tensor $\pi$ defines a degree two Poisson cohomology class, this being a generalization of the closedness of a symplectic form, 
it is no longer true that a (non-zero) Poisson structure on a compact manifold  be non-exact. That is, there may exist
a vector field $X$ such that 
\begin{equation}\label{eq:exact}
\cL_X\pi=-\pi.
 \end{equation}
The are some criteria for non-exactness of a Poisson structure:
\begin{itemize}
 \item A Poisson manifold whose induced Lie algebroid structure on its cotangent bundle is integrated by a compact source 1-connected 
 (symplectic) groupoid is non-exact \cite{CFM};
\item A Poisson manifold on a compact even dimensional manifold such that the top exterior power of the Poisson tensor vanishes transversely is non-exact
\cite{GMP,MO}.
\item A Poisson manifold with non-trivial modular class and a representative acting on multivector fields in a semisimple fashion is non-exact. This is a criterion
of Lu cited in \cite{W} and applied to show non-exactness of the Lie-Poisson structure on a compact connected semisimple Lie group introduced in \cite{LW}
and of the induced Bruhat Poisson structure on its manifold of full flags (and of a family of Poisson homogeneous structures
in the manifold of full flags introduced in \cite{Lu}).
\end{itemize}

Bruhat Poisson structures belong to the class of Poisson structures with a finite number of symplectic leaves. Members of this latter class of Poisson 
structures on compact manifolds may be exact.  
In this note we introduce the following elementary geometric criterion for non-exactness of Poisson structures with a finite 
number of symplectic leaves: 

\emph{If there exists a Poisson submanifold whose induced Poisson structure is non-exact then the ambient Poisson
structure is non-exact.}

The criterion is an immediate consequence of the following:
\begin{proposition}\label{pro:ex-inh} If a Poisson manifold with a finite number of leaves is exact, then any of its Poisson submanifolds 
must be exact. 
\end{proposition}
Our application is to {\bf toric Poisson structures} on (smooth projective) toric varieties:  Poisson structures invariant by the torus action whose symplectic leaves equal the finitely many torus orbits (see for instance \cite{Ca}). 
\begin{theorem}\label{thm:non-exact} Toric Poisson structures are non-exact.
\end{theorem}
In fact our criterion also allows to recover Lu's result on the non-exactness of Bruhat Poisson structures on manifolds of full flags,
and it also applies to Bruhat Poisson structures on arbitrary flag manifolds.

\section{Proofs of the results}

On a Poisson manifold $(M,\pi)$ we consider the partition $\mathcal{R}_\pi$ of $M$ defined by the following equivalence relation on points
of $M$: two points are related if they can be joined by a smooth curve on which the Poisson tensor has constant rank.

\begin{lemma}\label{lem:exact-partition} Let $(M,\pi)$ be a Poisson manifold and let $X\in \mathfrak{X}(M)$ such that 
$\cL_X\pi=-\pi$. Then the subsets of $\mathcal{R}_\pi$ are saturated by integral curves of $X$. 
\end{lemma}
\begin{proof} Let $x\in M$. On a small neighborhood $U$ of $x$ we have a flow map $\Psi:U\times (-\epsilon,\epsilon)\to M$ which integrates $X$.
The integration of (\ref{eq:exact}) is 
\begin{equation}\label{eq:exact-integ}\Psi_{t*}\pi=\mathrm{e}^{-t}\pi.
 \end{equation}
Because $\Psi_t$ is a diffeomorphism by (\ref{eq:exact-integ}) over the integral curve of $X$ through $x$ in $U$ the
rank of $\pi$ must be constant. Hence we deduce 
that the rank of $\pi$ is constant on integral curves. Thus integral curves of $X$ are contained in
the subsets of $\mathcal{R}_\pi$. 
\end{proof}

 A subset $R\in \mathcal{R}_\pi$ is saturated
by equidimensional symplectic leaves, but, in general, it may be far from being a submanifold of $M$. Union of subsets of $\mathcal{R}_\pi$
do not fit in general into submanifolds of $M$, either. For instance, in $\R^{3}$ with coordinates $x,y,z$ 
the bivector $\pi=f\partial_x\wedge\partial_y$ is Poisson for any 
$f\in C^{\infty}(\R^{3})$.  
The subsets of $\mathcal{R}_\pi$ are the connected components of $\R^{3}\backslash f^{-1}(0)$
and the smooth path connected components of $f^{-1}(0)$. Since any closed subset of $\R^{3}$ is the zero set of some smooth function, 
the subsets in $\mathcal{R}_\pi$ can be very complicated.
% of $f^{-1}(0)$. Any closed subset is the zero set of a smooth function.

\begin{lemma}\label{lem:exact-submanifold} Let $(M,\pi)$ be a exact Poisson manifold. If $N\subset M$ is a
submanifold\footnote{The submanifold is an embedded one of an initial immersed submanifold.}
which is saturated by subsets of $\mathcal{R}_\pi$, then it is a Poisson submanifold whose induced Poisson structure is also exact.
 \end{lemma}
\begin{proof}
Because $N$ is saturated by subsets of $\mathcal{R}_\pi$  it is saturated by symplectic leaves and thus the restriction $\pi_N$ of the Poisson 
tensor to $N$ is tangent to $N$, which makes $N$ into a Poisson submanifold. 
Lemma \ref{lem:exact-partition} implies that $N$ is also saturated  
by integral curves of $X$. 
Thus $X_N$ defined to be the restriction of $X$ to $N$ is tangent to $N$. 
Therefore from (\ref{eq:exact}) we deduce
\[\cL_{X_N}\pi_N=-\pi_N.\] 
\end{proof}

\begin{proof}[Proof of Proposition \ref{pro:ex-inh}]
Because $(M,\pi)$ has a finite number of symplectic leaves a subset $R\in \mathcal{R}_\pi$ is the union of the finitely many symplectic leaves 
of a given dimension, say  $d$. We claim that $R$ is a symplectic leaf: let $\gamma$ be a curve  contained
in $R$ and let $x\in \gamma$. In the local normal form (Weinstein splitting) of $\pi$ around $x$ each leaf of dimension $d$
corresponds to a countable subset in the factor of the transverse Poisson structure. Thus the union of leaves of dimension $d$ corresponds there to a countable 
subset of points, which is totally disconnected. 
Hence the projection onto this factor 
of the connected curve 
$\gamma$ around $x$ is  the point which corresponds to 
the leaf though $x$. Hence the curve $\gamma$ is locally contained in symplectic leaves, which implies that is must be contained in one symplectic leaf.
 
If $N\subset M$ is a Poisson submanifold then it must be a union of open subsets of the leaves of $\pi$. We cannot apply straight away Lemma
\ref{lem:exact-submanifold} because $N$ might not be union of symplectic leaves (subsets of $\mathcal{R}_\pi$). However, the
proof of the Lemma says that a vector field which satisfies (\ref{eq:exact}) is tangent to any symplectic leaf of $\pi$, and, thus, to any of its 
open subsets. Therefore the vector field is tangent to $N$ which implies that the induced Poisson structure is exact. 
\end{proof}
\begin{remark}\label{rem:loc-fin} Proposition \ref{pro:ex-inh} remains valid for Poisson manifolds whose partition in symplectic leaves is locally finite.
\end{remark}
\begin{remark} Proposition \ref{pro:ex-inh} is  generalizes the fact that any symplectic submanifold of an exact
 symplectic manifold is exact. The partition $\mathcal{R}$ must appear because vector fields (more generally, multivector fields) may not 
 be tangent to some Poisson submanifolds. For instance, (\ref{eq:exact}) is satisfied for the linear Poisson structure on $\R^{3}$
 \[\pi=x\partial_y\wedge\partial_z+y\partial_z\wedge\partial_x+z\partial_x\wedge\partial_y,\]
and the opposite of the Euler vector field.
 However, this vector field it is not tangent to the individual symplectic leaves ---the spheres--- whose induced Poisson structure is non-exact 
 for strictly positive radius.
\end{remark}

Toric Poisson structures on  projective toric varieties have a finite number of symplectic leaves and --- as we will recall --- they posses
Poisson submanifolds of dimension two. A Poisson surface with a 
finite number of leaves can be exact. For instance, the Poisson structure $\pi=(x^{2}+y^{2})^{2}\partial_x\wedge \partial_y$ on $\R^{2}$ is exact; equation 
(\ref{eq:exact}) is satisfied for the opposite of the Euler vector field. 
Moreover, it extends to an exact Poisson structure on the sphere
with two symplectic leaves (one uses as change of coordinates the radial inversion).

The next result --- which we include for the sake of completeness --- is well-known to experts.
\begin{lemma}\label{lem:folk} Let $\pi=f\partial_x\wedge\partial_y$ be a rotationally invariant Poisson structure in a disk $D\subset \R^{2}$
of some radius $r_0\in (0,\infty]$,
where $f\in C^{\infty}(D)$ vanishes at the origin and with order 2.
 Then $\pi$ is non-exact.
\end{lemma}
\begin{proof} Let $r\partial_r,\partial_\theta$ denote the Euler and rotational vector fields.
The invariance of $\pi$ implies that $f(x,y)=g(r^{2})$, $x^{2}+y^{2}=r^{2}$, 
$g\in C^{\infty}([0,r_0))$.
The vanishing at the origin at order 2 implies that we can factor 
\[g(r^{2})=r^{2}h(r^{2}),\quad h\in C^{\infty}([0,r_0)),\quad h(0)\neq 0,\]
and, therefore, $\pi=hr\partial_r\wedge \partial_\theta$.
Let us assume that there exists $X\in \mathfrak{X}(D)$ such that $\cL_X\pi=-\pi$. Because $\pi$ is rotationally invariant we may assume without loss
of generality that $X$ is rotationally invariant. Equivalently,
\[X=a(r^{2})r\partial_r+b(r^{2})\partial_\theta,\quad a,b\in C^{\infty}([0,r_0)).\]
From the Leibniz rule for the Lie derivative and the invariance of $X$ we deduce 
\[\begin{split} \cL_{X}(hr\partial_r\wedge \partial_\theta) & =(\cL_X hr\partial_r)\wedge \partial_\theta+hr\partial_r\wedge \cL_X \partial_\theta=
\cL_{ar\partial_r} (hr\partial_r)\wedge \partial_\theta=\\
&= (a\cL_{r\partial_r}h-h\cL_{r\partial_r}a)r\partial_r\wedge \partial_\theta=2r^{2}(h'a-ha')r\partial_r\wedge \partial_\theta.  
\end{split}
\]
Hence exactness of $\pi$ forces  $h(0)=0$, which contradicts the assumption on $f$.
\end{proof}
% \begin{remark}\label{rem:exact} 
%  \end{remark}

\begin{proof}[Proof of Theorem \ref{thm:non-exact}]
Let $(X,\mathbb{T})$ be a projective\footnote{It suffices to assume that the smooth toric variety has a fixed point.} toric variety.
Let $x\in X$ be a fixed point for the action of $\mathbb{T}$ and let us consider the corresponding toric chart centered at $x$. 
This amounts to certain group isomorphism of $\mathbb{T}$ with $(\C^{*})^{n}$ together with an equivariant biholomorphism
whose image is $\C^{n}$ with the standard $(\C^{*})^{n}$-action and which sends $x$ to the origin. Therefore any 
complex axis $\C\subset \C^{n}$ corresponds to a toric subvariety $(Y,\mathbb{T}_Y)$ of $(X,\mathbb{T})$ biholomorphic to 
$(\C,\C^{*})$. 

The symplectic leaves of the toric Poisson structure $\pi$ on $X$ are the $\mathbb{T}$-orbits. Therefore the submanifold
$Y$ is in the hypotheses of Proposition
\ref{pro:ex-inh} and thus it has an induced Poisson structure $\pi_Y$. We use the toric chart to induce a Poisson 
structure $\pi_\C$ on $\C$ which is $\C^*$-invariant. This 
implies that there exists  a constant Poisson structure on the Lie algebra of $\C^{*}$ and $\pi_\C$ is obtained by replacing the vectors 
in the formula of the constant bivector by the fundamental vector fields of the $\C^{*}$-action. Equivalently,
\[\pi_\C=kr\partial_r\wedge \partial_\theta,\,k\in \R\backslash 0.\]
Hence it is a rotationally invariant Poisson structure which vanishes  at the origin and with order 2. By Lemma \ref{lem:folk}
$\pi_\C$ is not exact. Thus by Proposition \ref{pro:ex-inh} $(X,\pi)$ cannot be exact. 
\end{proof}
\begin{remark} One can show that each individual symplectic leaf of a toric Poisson structure is exact.
 \end{remark}

\begin{remark}[Bruhat Poisson structures]
Let $K$ be a compact semisimple Lie group, let $G$ be its complexification, $T\subset K$ a maximal torus and $B\subset G$ a the positive
Borel subgroup determined by the fixed root ordering. The inclusion of $K$ into $G$ induces a diffeomorphism between the two models of the manifold of full flags
\[ K/T\to G/B.\]
If $P$ is a parabolic subgroup then we have the projection between flag manifolds
\[G/B\to G/P.\]
Associated to the root ordering and a choice of root vectors there is a Poisson-Lie group structure $\pi$ on $K$ which is projectable for the submersion
\[q:K\to G/P\]
to the so-called Bruhat Poisson structure $\pi_{G/P}$ whose symplectic leaves are the Bruhat cells \cite{LW}. Lu has proved that the
Poisson structures $\pi$ and $\pi_{K/T}$ are not exact (and similarly for a family of Poisson homogeneous structures on $K/T$ introduced in \cite{Lu}).

We can recover Lu's result using our methods: let $\alpha$ be a simple root not in the subset of simple roots which defines $P$. There is a commutative diagram 
\[\xymatrix{ 
(\mathrm{SU}(2),\pi_\alpha)\ar[r]^{\phi_\alpha}\ar[d] &  (K,\pi)\ar[d]^{q}\\
 (\mathrm{SU}(2)/\S^{1},[\pi_{\alpha}])\ar[r]^{[\phi_\alpha]}) & (\mathbb{CP}^{1}_\alpha,\pi_{\mathbb{CP}^{1}_\alpha})\subset (G/P,\pi_{G/P})}.
\]
where $\phi_\alpha$ is a monomorphism between Poisson Lie groups described in \cite{Lu2}.  The vertical arrows are Poisson morphisms \cite{LW}.  From properties of the Bruhat decomposition it follows 
that the restriction of $q$ to the image of $\phi_\alpha$ is a submersion over the image of $[\phi_\alpha]$,
 which is the closure of the Bruhat cell in $G/P$ defined by $\alpha$, and which is diffeomorphic to a complex projective line.
Therefore this complex projective line is a Poisson submanifold of $(G/P,\pi_{G/P})$  Poisson
diffeomorphic to $(\mathrm{SU}(2)/\S^{1},[\pi_{\alpha}])$. The Lie-Poisson structure $\pi_\alpha$ is a multiple of the 
Lie-Poisson structure for $\mathrm{SU}(2)$. It is invariant under the left and right actions of $\S^{1}$. Hence 
$[\pi_{\alpha}]$ is a rotationally invariant Poisson structure vanishing at one point with order two (see for instance the Appendix in \cite{She}).
Therefore by  Proposition \ref{pro:ex-inh} $(G/P,\pi_{G/P})$ cannot be exact.

% An immediate consequence is the non-exacness of the Poisson-Lie group $(K,\pi)$. Consider the Poisson submersion
% $q:(K,\pi)\to (K/T,\pi_{K/T})$.
% The Poisson structure $\pi$ is (right) $T$-invariant. If we have $X$ a vector field such that  $\cL_{X}\pi=\pi$, we may assume that it is
% $T$-invariant. In other words, both
% $\pi$ and $X$ are projectable which implies 
% \[\cL_{q_*X}\pi_{K/T}=\pi_{K/T}.\]
% Hence $\pi$ cannot be exact.
\end{remark}

\begin{remark} There are Poisson structures ---other that Toric Poisson and Bruhat Poisson structures---  for which the partition $\mathcal{R}_\pi$ 
is well-behaved and gives rise to natural Poisson submanifolds. One such instance of Poisson submanifold is  the so-called singular locus of
 a Poisson bivector whose top exterior power vanishes transversely, and, more generally,
 the singular locus of the Poisson structures 
described in \cite{GMW} and \cite{MS}. On a compact manifold the non-exactness of a Poisson tensor 
whose top exterior power vanishes transversely is deduced from deep results on its Poisson cohomology \cite{MO,GMW}. One
could approach the problem in the spirit of this note and prove first the non-exactness of its singular locus. This strategy would also 
be valid to analyze the non-exactness of the more general Poisson structures described in \cite{GMW} and \cite{MS}.
\end{remark}

\end{document}